\documentclass[a4paper,11pt]{article}

\usepackage{amsmath,amssymb,amsthm,mathtools}
\usepackage{newtxtext,newtxmath}
\usepackage{bm,booktabs}
\usepackage{geometry}

\usepackage{natbib}
\newtheorem{teiri}{Theorem}
\newtheorem{hodai}{Lemma}
\newtheorem{prop}{Proposition}
\newtheorem{kei}{Corollary}
\newtheorem{rei}{Example}

\title{Singularities of Multidimensional Item Response Models and the Rank of the Fisher Information Matrix}
\author{Joe Suzuki}

\begin{document}
\maketitle

\begin{abstract}
Statistical models in item response theory are known to be singular. This singularity manifests itself as a loss of rank in the Fisher information matrix, and the usual theories underlying AIC and BIC, which assume regularity, cannot be applied without further justification. Chen and colleagues analyzed this degeneracy and nonregular likelihood ratio tests, but their analysis does not identify all the directions that produce degeneracy or their dimensions. In this study, we analyze the entire kernel of the Fisher information matrix as a linear space for multidimensional item response models, separating the degeneracy due to rotations from that due to an over-specified latent dimension without double counting. We determine the dimension of each component and the exact rank of the information matrix, and prove a formula for computing, from the item parameters, the dimension of the degeneracy that remains after these components are removed. These results go beyond establishing the existence of singularity to identifying its causes and extent, providing a foundation for inference and model comparison in IRT.

\end{abstract}

\section{Introduction}
Item response theory (IRT) is a framework of probabilistic models that explain responses to individual test or questionnaire items in terms of respondents' abilities or traits and the properties of the items. Because abilities and traits cannot be observed directly, they are treated as latent variables and estimated from responses to multiple items. IRT is used to measure educational achievement, assess traits using psychological scales, and measure health status from patient-reported outcomes\citep{EmbretsonReise2000,CellaEtAl2010PROMIS}. It is also used in adaptive testing, where the next item is selected according to the responses already obtained\citep{Reckase2009MIRT}.

Multidimensional IRT is used when several abilities or traits contribute to responses. The multidimensional two-parameter logistic (2PL) model considered in this paper describes binary responses, such as correct or incorrect answers and affirmative or negative responses. Let $M$ be the number of items, $i$ the number of ability dimensions, $\theta\in\mathbb R^i$ a respondent's latent ability, and $X_\ell\in\{-1,1\}$ the response to item $\ell$. Then
\[
P(X_\ell=1\mid\theta)=\sigma(a_\ell^\top\theta-b_\ell),
\qquad
\sigma(u)=\frac{1}{1+e^{-u}}
\]
The discrimination vector $a_\ell$ describes how each ability direction affects the response, while the difficulty parameter $b_\ell$ controls how difficult it is to obtain an affirmative response. Responses to different items are assumed to be independent conditional on ability.

Because ability is unobserved, item parameters are estimated using the marginal response distribution obtained by integrating over the ability distribution\citep{BockAitkin1981}. In this paper, the ability distribution is standard normal. The Fisher information matrix of the marginal distribution describes how well changes in the item parameters can be detected from responses. When this matrix is singular, there is a nonzero parameter direction along which the first-order changes in all marginal response probabilities vanish. The set of these directions is the kernel of the Fisher information matrix.

Degeneracy of the Fisher information matrix reflects singularity of the statistical model. This problem is not specific to IRT. For singular statistical models, the usual AIC correction for generalization error and the BIC approximation to the Bayesian marginal likelihood cannot be justified on the same grounds as for regular models\citep{Watanabe2010WAIC,DrtonPlummer2017sBIC}. Before applying these criteria mechanically, it is therefore necessary to examine whether degeneracy remains after identifiability constraints are imposed.

Let $A\in\mathbb R^{M\times i}$ be the matrix whose rows are the discrimination vectors, and let $j$ be the rank of its true value $A_*$. Under standard normal ability, rotating the ability coordinates leaves the marginal response distribution unchanged, and parameter directions corresponding to rotations belong to the kernel of the Fisher information matrix. When $j<i$, further degeneracy arises from ability directions that are unused at the true value. These two causes do not, however, account for all degeneracy. Even after the effects of rotations are removed and $i=j$, degeneracy may remain depending on the correspondence between items and ability directions. This paper aims to separate these causes and determine the entire kernel of the Fisher information matrix, including the degeneracy that remains.

\subsection{Previous Research on Item Response Models and Singularity}
In multidimensional IRT, models are compared using AIC\citep{Akaike1974Identification} and BIC\citep{Schwarz1978Dimension} to select the number of ability dimensions and their correspondence with items. Examples include comparisons of one-factor, correlated-factor, and bifactor models for ordinal responses\citep{ImmekusSnyderRalston2019MIRT}, and comparisons of the number of factors for nominal responses\citep{RevueltaMaydeuXimenez2020Nominal}. Comparisons of predictive performance in multidimensional IRT using DIC, WAIC, and cross-validation have also been studied\citep{FujimotoFalk2024ModelFit}.

Identifiability concerns whether item parameters can be distinguished from the response distribution. Under a known nondegenerate ability distribution and positive discrimination parameters, \citet{SanMartinRolinCastro2013} established identifiability of the one-factor 2PL model with at least three items. For the standard probit bifactor model, which has a general factor and factors specific to item groups, \citet{FangGuoXuYingZhang2021} gave necessary and sufficient conditions for identifiability in terms of the number of items and the loading structure. These results cannot, however, be applied directly to the logistic model with normal ability integrated out.

Nonregularity due to an over-specified latent dimension has been studied in likelihood ratio tests for the number of factors. \citet{HayashiBentlerYuan2007}, building on the observations of \citet{TakaneVanDerHeijdenBrowne2003}, discussed degeneracy of the information matrix due to over-specification of the number of factors in exploratory factor analysis with continuous responses. For binary item response models, \citet{ChenMoustakiZhang2020} analyzed degeneracy of the Fisher information matrix and nonregularity of likelihood ratio tests when a one-factor model is embedded in a two-factor model.

These results do not determine the entire kernel of the Fisher information matrix or its exact dimension for general $i,j$. In particular, considering only rotations and an over-specified dimension does not capture degeneracy arising from other causes. This paper classifies the kernel and computes the rank, including the remaining degeneracy that depends on the item structure. The questions left open by \citet{ChenMoustakiZhang2020} are discussed in Section 2.2.

\subsection{Contributions}
This paper makes two contributions.

First, for general $1\leq j\leq i\leq M$, we give the following orthogonal decomposition of the kernel of the Fisher information matrix $I_*$ at the true value:
\[
\ker I_*=\mathcal R_{i,j}\oplus\mathcal N_{i,j}\oplus\mathcal E_{i,j}
\]
Here, $\mathcal R_{i,j}$ consists of directions due to rotations, $\mathcal N_{i,j}$ consists of directions associated with the over-specified latent dimension that cannot be represented by rotations, and $\mathcal E_{i,j}$ consists of nonrotational directions remaining within the $j$-dimensional model. We determine the dimension of each component and, writing $\kappa_*:=\dim\mathcal E_{i,j}$, prove that
\[
\operatorname{rank}I_*=Mj+M-\frac{j(j-1)}2-\kappa_*
\]
This counts the causes and dimensions of degeneracy without overlap and shows that, for fixed true item parameters, the rank does not depend on the candidate dimension $i$.

Second, we compute the dimension $\kappa_*$ of the remaining degeneracy from the item parameters. After a rotation, write $A_*=[B,0]$, $B\in\mathbb R^{M\times j}$, and let $b_*$ be the true difficulty vector. We explicitly construct a matrix $L_j(B,b_*)$ with $M$ columns and prove that
\[
\kappa_*=M-\operatorname{rank}L_j(B,b_*)
\]
Furthermore, when the structure of $B$ is such that each item corresponds to at most one ability direction, we show, for arbitrary difficulties, that $\kappa_*$ equals the number of ability directions corresponding to one or two items. For this item structure, having at least three items corresponding to each ability direction is necessary and sufficient for eliminating the remaining degeneracy.

Section 2 defines the model and the Fisher information matrix. Section 3 gives the orthogonal decomposition of the kernel and the rank formula, and Section 4 gives the formula for computing $\kappa_*$ and evaluates it for particular item structures. Section 5 discusses the results and future research.

\section{Preliminaries}
This section provides the background needed to understand the main results in Sections 3 and 4. We first define the Fisher information matrix and then explain the questions arising from previous research and the motivation for this study.

\subsection{Singularity of Statistical Models and the Fisher Information Matrix}

Let $M$ be the number of items and $i$ the latent dimension allowed by the candidate model, where $1\leq i\leq M$. We express the candidate model $\mathcal M_i$ using ability $\theta\in\mathbb R^i$ and a loading matrix $A\in\mathbb R^{M\times i}$. The ability distribution is $N(0,I_i)$, whose density is denoted by $\varphi_i$, where $I_i$ is the identity matrix of size $i$.
Let $X=[X_1,\ldots,X_M]^\top\in\{-1,1\}^M$ be the response vector of a single examinee, with $X_\ell=1$ denoting a correct response and $X_\ell=-1$ an incorrect response. Write the $\ell$th row of $A$ as $a_\ell^\top\in {\mathbb R}^i$, and let $b=[b_1,\ldots,b_M]^\top$ be the difficulty vector. Conditional on ability $\theta$, responses to different items are independent, and the conditional response probability is
\[
p(x\mid\theta,A,b)=\prod_{\ell=1}^M\sigma\{x_\ell(a_\ell^\top\theta-b_\ell)\},\qquad \sigma(u)=\frac1{1+e^{-u}}
\]
We define the marginal response probability obtained by integrating out ability as
\begin{equation}\label{eq:irt_marg_response}
p_i(x\mid A,b):=\int_{\mathbb R^i}
p(x\mid\theta,A,b)
\varphi_i(\theta)\,d\theta
\end{equation}
The likelihood used below is based on this marginal distribution.

Let $A_*$ be the true loading matrix and $b_*$ the true difficulty vector, and set
\[
j=\operatorname{rank}A_*,\qquad 0\leq j\leq i
\]
Here, $i$ is the latent dimension allowed by the candidate model, and $j$ is the number of independent ability directions in the true loading matrix. For example, when $i=2$ and $j=1$, the fitted model allows two ability directions, but only one direction contributes to responses at the true value. In the next section, we separate the nonidentifiability due to rotations that is common to $j<i$ and $j=i$, and then examine the additional degeneracy arising when $j<i$.

Let $w=[w_1,\ldots,w_d]^\top$ be the parameter vector obtained by arranging all entries of $A,b$ in a fixed order, where $d=Mi+M$. Let $w_*$ be the vector obtained by arranging the entries of the true parameters $[A_*,b_*]$ in the same order. We also use $p_i(x\mid w)$ to denote the same function as $p_i(x\mid A,b)$, and write the true distribution as $q_x=p_i(x\mid w_*)$. For every response pattern $x$, we have $q_x>0$. Define the Kullback--Leibler divergence from the true marginal distribution by $$K(w):=\sum_{x\in\{-1,1\}^M}q_x\log\frac{q_x}{p_i(x\mid w)}$$

Here, the statistical model $\{p_i(\cdot\mid w)\}$ is called regular if the following three conditions hold.
\begin{enumerate}
\item In a sufficiently small neighborhood of the true value $w_*$, the only point minimizing $K(w)$ is $w_*$.
\item $w_*$ is an interior point of the parameter space.
\item The Fisher information matrix at $w_*$, defined below, is nonsingular, that is, invertible.
\end{enumerate}
If the model is not regular, it is called singular.

Below, we define the score, the vector of partial derivatives of the log likelihood with respect to the individual parameter entries, by
\begin{equation}\label{eq801}
s_i(x;w):=
\frac{\nabla p_i(x\mid w)}{p_i(x\mid w)}
\end{equation}
and denote the score at the true value by
$s_i(x;w_*)$.
Taking the expectation under the true marginal distribution, we define the Fisher information matrix as
\begin{equation}\label{eq:irt_marg_fisher_def}
I_*:=\sum_{x\in\{-1,1\}^M}q_xs_i(x;w_*)s_i(x;w_*)^\top
=
\sum_{x\in\{-1,1\}^M}\frac{1}{p_i(x\mid w_*)}
\nabla p_i(x\mid w_*) \nabla p_i(x\mid w_*)^\top  
\in {\mathbb R}^{d\times d}
\end{equation}
The Fisher information matrix is said to be degenerate when its rank is less than $d=Mi+M$.


When $j\geq1$, consider the marginal model in \eqref{eq:irt_marg_response} with candidate dimension $j$ and loading matrix $B\in {\mathbb R}^{M\times j}$. Here, $B$ is the matrix at the true value in \eqref{eq:irt_marg_true_loading}. Set $v_*=[B,b_*]$, and denote the score and Fisher information matrix of this model by $s_j(x;v_*)$ and $J_*$, respectively. Specifically, $s_j(x;v_*)$ is the vector of partial derivatives of the log marginal probability with respect to the $Mj+M$ entries of $B,b$, evaluated at $v_*$, and
\begin{equation}\label{eq806}
J_*:=\sum_xq_xs_j(x;v_*)s_j(x;v_*)^\top
=\sum_{x\in\{-1,1\}^M}\frac{1}{p_j(x\mid v_*)}
\nabla p_j(x\mid v_*)\nabla p_j(x\mid v_*)^\top
\in\mathbb R^{(Mj+M)\times(Mj+M)}
\end{equation}
If the last $i-j$ columns are zero, integration over the corresponding ability components gives 1, so that
\begin{equation}\label{eq:irt_marg_submodel_identity}
p_i(x\mid[B,0],b)=p_j(x\mid B,b)
\end{equation}
and the true response probabilities of the two models are the same $q_x$. When $j=0$, define
\[
p_0(x\mid b):=\prod_{\ell=1}^M\sigma(-x_\ell b_\ell)
\]
and let $J_*$ at $v_*=b_*$ be the Fisher information matrix with respect to $b$. In this case, entries corresponding to $B$ are interpreted as absent. Moreover, since ability has a standard normal distribution, $Q\theta$ and $\theta$ have the same distribution for any orthogonal matrix $Q$ of size $i$. Therefore,
\begin{equation}\label{eq:irt_marg_rotation_invariance}
p_i(x\mid AQ,b)=p_i(x\mid A,b)
\end{equation}
By choosing an orthonormal basis of the null space of $A_*$ as the last $i-j$ coordinate directions, we may assume without loss of generality that
\begin{equation}\label{eq:irt_marg_true_loading}
A_*=[B,0],\qquad B\in\mathbb R^{M\times j},\quad \operatorname{rank}B=j
\end{equation}

\begin{hodai}\label{hodai1}
Let $0\leq j\leq i\leq M$, $i\geq1$, and express the true loading matrix as in \eqref{eq:irt_marg_true_loading}. If the parameters are reordered as the first $j$ columns, the difficulties, and the last $i-j$ columns, then
$$
I_*=\begin{bmatrix}J_*&0\\0&0_{M(i-j)}\end{bmatrix},\qquad \operatorname{rank}I_*=\operatorname{rank}J_*
$$
Here, $0_{M(i-j)}$ is the $M(i-j)\times M(i-j)$ zero matrix, and this block is absent when $j=i$.
\end{hodai}
Proof: See Appendix A.

\subsection{The Results of Chen and Colleagues and the Remaining Questions}
In the notation of the preceding subsection, the specific comparison of ability dimensions in binary item response models considered by \citet{ChenMoustakiZhang2020} is the case $i=2,j=1$. Example 2(a) of their paper considers a test of whether responses can be explained by one ability direction or whether a second direction is needed. They write the true loading matrix as $A_*=[B,0]$, $B\in\mathbb R^{M\times1}$, $\operatorname{rank}B=1$, and show that the Fisher information matrix is degenerate because the scores with respect to the second column are zero. They also use tangent-cone analysis to give the limiting distribution of the likelihood ratio statistic in this nonregular setting. However, establishing this degeneracy alone does not identify its causes and extent. Three questions remain.

First, the degeneracy due to rotations has not been quantified. As stated in the preceding subsection, rotating the ability coordinates and changing the loading matrix accordingly leaves the marginal response distribution unchanged. Chen and colleagues impose the constraint $a_{12}=0$ in the two-factor model. When $a_{*11}\ne0$, this aligns the discrimination vector of the first item with the first coordinate direction and removes the continuous rotational degree of freedom in this case. For general $i,j$, however, the dimension of the kernel produced by rotations and the resulting loss of rank of the Fisher information matrix have not been determined. Removing rotations by a particular constraint and determining their effect on degeneracy are different tasks.

Second, degeneracy due to an over-specified latent dimension has been examined explicitly only for $i-j=1$. Lemma 1(2) and its proof in their paper use the fact that the second ability direction does not contribute to responses at the true value and that the corresponding standard normal distribution has mean zero to obtain
\[
\left.\frac{\partial p_2(x\mid A,b)}{\partial a_{\ell2}}\right|_{(A_*,b_*)}=0,\qquad \ell=2,\ldots,M
\]
Thus, even after rotations are restricted, $M-1$ independent directions remain in the kernel. However, the increase in the kernel dimension when there are several extra ability directions has not been determined. Nor has the overlap between these directions and those due to rotations been classified. Consequently, this result does not determine the loss of rank for general $i,j$.

Third, and more fundamentally, degeneracy other than that due to rotations and an over-specified latent dimension has not been examined. Showing that the scores with respect to the second column vanish does not establish that the scores for the first-column discrimination parameters and the difficulties are linearly independent. If these scores are linearly dependent, there are directions in which joint changes to the parameters of the existing ability directions and to the difficulties produce no first-order change in the marginal response probabilities. In the notation of the preceding subsection, identifying the zero block corresponding to the extra ability directions does not determine whether the remaining matrix $J_*$ is degenerate. An explanation based only on the two causes overlooks this issue.

Indeed, let $M=3$, $i=j=1$, $A_*=[a,0,0]^\top$, $a\ne0$, and $b_*=0$. In this case, there is neither a continuous rotational degree of freedom nor a difference between the candidate dimension and the rank of the true loading matrix. Nevertheless, the symmetry of $Z\sim N(0,1)$ and the identity $\sigma(z)+\sigma(-z)=1$ imply that, for every $x\in\{-1,1\}^3$,
\[
p_1(x\mid[a,0,0]^\top,0)=\frac14{\mathbb E}_Z[\sigma(x_1aZ)]=\frac18
\]
Changing $a$ therefore leaves the marginal distribution unchanged, so its score is zero and the Fisher information matrix is degenerate. Thus, it is incorrect to conclude that degeneracy disappears once rotations are removed and the candidate dimension is made equal to the rank of the true loading matrix. Under what conditions does degeneracy remain after rotations and the over-specified latent dimension are removed? This question is not answered by the analysis of Chen and colleagues.

\section{Orthogonal decomposition of the kernel of the Fisher information matrix}

Throughout this section, assume $1\leq j\leq i\leq M$ and set $r:=i-j$.
Let $w$ denote the parameter vector obtained by arranging the components of $(A,b)$ in a fixed order, and write $p_i(x\mid w)$ for the marginal response probability of the response pattern $x\in\{-1,1\}^M$. Also, let $w_*=(A_*,b_*)$.

First, replace $Q$ in (\ref{eq:irt_marg_rotation_invariance})
by a function $Q(t)$
that takes values in the orthogonal matrices of size $i$, is differentiable in a neighborhood of $t=0$, and satisfies $Q(0)=I_i$.
Set
$\Gamma:=Q'(0)$. Differentiating $Q(t)^\top Q(t)=I_i$ gives $\Gamma^\top=-\Gamma$. Partitioning $\Gamma$ according to $j$ and $r$, we can write
\[
\Gamma=\begin{bmatrix}\Omega&T\\-T^\top&\Omega_0\end{bmatrix},\qquad
\Omega^\top=-\Omega,\quad T\in\mathbb R^{j\times r},\quad \Omega_0^\top=-\Omega_0
\]
Moreover,
let $w(t)$ be the parameter vector corresponding to $[A_*Q(t),b_*]$, and set $$u:=\left.\frac{dw(t)}{dt}\right|_{t=0}$$
This is precisely the vector representation of
$$\left.\frac{d}{dt}[A_*Q(t),b_*]\right|_{t=0}=[A_*\Gamma,0]=[B\Omega,BT,0]
$$
Here, (\ref{eq:irt_marg_rotation_invariance})
means that its left-hand side is constant, regardless of the orthogonal matrix $Q$.
Therefore,
differentiating (\ref{eq:irt_marg_rotation_invariance}) with respect to $t$ and setting $t=0$ gives
\[
0
=\left.\frac{dp_i(x\mid w(t))}{dt}\right|_{t=0}=\nabla p_i(x\mid w_*)^{\mathsf T}u
\qquad x\in\{-1,1\}^M
\]
Multiplying (\ref{eq:irt_marg_fisher_def}) by $u$ on the right then gives
$I_*u=0$. Thus, the vector $[B\Omega,BT,0]$ generated by this rotation belongs to the kernel of $I_*$. Differentiating a rotation involving only two coordinates gives a skew-symmetric matrix whose two corresponding off-diagonal entries are $1,-1$ and whose other entries are zero. Every skew-symmetric matrix is a linear combination of these matrices. Since the kernel is a linear space, the conclusion above holds for every $\Omega,T$.
\[
\mathcal R_{i,j}:=\{[B\Omega,BT,0]\in\mathbb R^{Mi+M}:\Omega\in\mathbb R^{j\times j},\ \Omega^\top=-\Omega,\ T\in\mathbb R^{j\times r}\}
\]
With this definition, $\mathcal R_{i,j}\subseteq\ker I_*$.

Also, for $C\in {\mathbb R}^{M\times r}$, marginalization under the standard normal distribution gives
$$p_i(x\mid [B,tC],b_*)=p_i(x\mid [B,-tC],b_*)$$
Differentiating both sides with respect to $t$ at zero gives quantities of equal magnitude and opposite sign. Hence,
\[
\left.\frac{d}{dt}p_i(x\mid [B,tC],b_*)\right|_{t=0}=0
\qquad
\]
By the same argument, every vector of the form $[0,C,0]$ belongs to the kernel of $I_*$.


Each column of $BT$ belongs to the column space of $B$. We therefore define the directions in its orthogonal complement within the last $r$ columns by
\[
\mathcal N_{i,j}:=\{[0,C,0]\in\mathbb R^{Mi+M}:C\in\mathbb R^{M\times r},\ B^\top C=0\}
\]
Every $C\in\mathbb R^{M\times r}$ has the unique decomposition
\begin{equation}\label{eq14}
C=BT_C+C_{\perp},\qquad
T_C=(B^\top B)^{-1}B^\top C,\qquad
C_\perp=C-BT_C
\end{equation}
Indeed, $B^\top C_\perp=0$, and the decomposition is unique because $B^\top B$ is nonsingular.

The $j$-dimensional model itself may also have kernel directions other than rotations. Define
\[
\mathcal E_{i,j}:=
\{[D,0,c]\in\mathbb R^{Mi+M}:[D,c]\in\ker J_*,\
B^\top D=D^\top B\}
\]
The condition
$B^\top D=D^\top B$
is equivalent to requiring that,
for every skew-symmetric matrix $\Omega$,
$D$ and $B\Omega$ are orthogonal when viewed as column vectors of length $Mj$,
that is, ${\rm tr}(D^\top B\Omega)=0$.
Indeed, for $S=D^\top B$,
$${\rm tr}(S\Omega)=
\sum_{1\leq r\leq j}\sum_{1\leq s\leq j}S_{rs}\Omega_{sr}
=
\sum_{1\leq r\leq j-1}\sum_{r+1\leq s\leq j}(S_{sr}-S_{rs})\Omega_{rs}=0
$$
Since $\Omega_{rs}$ (${1\leq r\leq j-1}$, ${r+1\leq s\leq j}$) can be chosen arbitrarily, $S$ must be symmetric (sufficiency is immediate).

We now obtain the following theorem.

\begin{teiri}[Orthogonal decomposition of the kernel]\label{prop:irt_rotation_separation}
Assume $1\leq j\leq i\leq M$. The linear spaces $\mathcal R_{i,j}, \mathcal N_{i,j}, \mathcal E_{i,j}
$ satisfy
\begin{equation}\label{eq:irt_full_kernel_separation}
\ker I_*=\mathcal R_{i,j}\oplus\mathcal N_{i,j}\oplus\mathcal E_{i,j}
\end{equation}
Here, $\oplus$ means that the subspaces are mutually orthogonal and the decomposition is unique. Moreover,
\begin{equation}\label{eq:irt_rotation_direct_sum_dimensions}
\dim\mathcal R_{i,j}=\frac{j(j-1)}2+j(i-j),\qquad \dim\mathcal N_{i,j}=(M-j)(i-j)
\end{equation}
Thus, setting $\kappa_*:=\dim\mathcal E_{i,j}$ gives
\begin{equation}\label{eq:irt_exact_rank_with_residual}
\operatorname{rank}I_*=Mj+M-\frac{j(j-1)}2-\kappa_*
\end{equation}
The quantity $\kappa_*$ is the dimension of the kernel remaining in the $j$-dimensional model after rotation directions have been removed, and does not depend on the candidate dimension $i$.
\end{teiri}

\begin{proof}
By Lemma~\ref{hodai1}, $[D,C,c]\in\ker I_*$ is equivalent to $[D,c]\in\ker J_*$, and $C$ is unrestricted. In the inclusion $\mathcal R_{i,j}\subseteq\ker I_*$ established above using $\Gamma=Q'(0)$, setting $T=0$ shows that $[B\Omega,0,0]\in\ker I_*$ for every skew-symmetric matrix $\Omega$. Hence, Lemma~\ref{hodai1} gives
\[
[B\Omega,0]\in\ker J_*
\]
Decomposing the finite-dimensional inner product space $\ker J_*$ into this space of rotation directions and its orthogonal complement, every $[D,c]\in\ker J_*$ can be written uniquely as
\[
[D,c]=[B\Omega,0]+[D_\perp,c],\qquad
\Omega^\top=-\Omega,\quad
[D_\perp,c]\in\ker J_*,\quad
B^\top D_\perp=D_\perp^\top B
\]
Using Eq.~\eqref{eq14} for $C$, we obtain
\[
[D,C,c]=[B\Omega,BT_C,0]+[0,C_\perp,0]+[D_\perp,0,c]
\]
The three terms on the right belong to $\mathcal R_{i,j}$, $\mathcal N_{i,j}$, and $\mathcal E_{i,j}$, respectively. Conversely, every element of $\mathcal R_{i,j}$ belongs to the kernel of $I_*$ by Eq.~\eqref{eq:irt_marg_rotation_invariance}, and every element of $\mathcal N_{i,j}$ or $\mathcal E_{i,j}$ also belongs to the kernel of $I_*$ by its definition and Lemma~\ref{hodai1}.

We have already shown that the vector representations of $BT_C,C_\perp$ are orthogonal,
as are those of $D_\perp,B\Omega$.
Since $\mathcal N_{i,j}$ and $\mathcal E_{i,j}$ affect different components, they are orthogonal. This proves Eq.~\eqref{eq:irt_full_kernel_separation}.

Since $B$ has full column rank, the map $[\Omega,T]\mapsto[B\Omega,BT]$ is injective. A skew-symmetric matrix $\Omega$ has $j(j-1)/2$ free entries, and $T$ has $jr$ free entries. Therefore,
\[
\dim\mathcal R_{i,j}=\frac{j(j-1)}2+jr
\]
The column space of $B$ is a $j$-dimensional subspace of $\mathbb R^M$, so its orthogonal complement has dimension $M-j$. The condition $B^\top C=0$ means that each column of $C$ belongs to this orthogonal complement. Since the $r$ columns can be chosen independently,
\[
\dim\mathcal N_{i,j}=(M-j)r
\]
Adding these dimensions gives
\[
\dim\ker I_*=\frac{j(j-1)}2+jr+(M-j)r+\kappa_*
=\frac{j(j-1)}2+M(i-j)+\kappa_*
\]
Subtracting this from $Mi+M$ yields Eq.~\eqref{eq:irt_exact_rank_with_residual}.
\end{proof}

When $j=0$, $J_*$ is the Fisher information matrix of a model with independent Bernoulli responses and is nonsingular. Hence, Lemma~\ref{hodai1} gives $\operatorname{rank}I_*=M$.
Moreover, by Lemma~\ref{hodai1} and the nonsingularity of $J_*$, the kernel consists only of directions associated with the additional ability coordinates. We therefore define $\kappa_*=0$ in this case. Eq.~\eqref{eq:irt_exact_rank_with_residual} also holds in this case.

\section{Calculation of $\kappa_*$}

In the preceding section, we decomposed the kernel of the Fisher information matrix for the marginal 2PL model into rotation directions, directions corresponding to excess ability dimensions, and the remaining directions. In this section, we obtain the dimension $\kappa_*$ of $\mathcal E_{i,j}$, defined in the preceding section, from the rank of a matrix constructed from the true item parameters.
Let $1\leq j\leq i\leq M$, and let $B\in\mathbb R^{M\times j}$ and $b_*\in\mathbb R^M$ be the true loading matrix and difficulty vector, respectively. Assume that $\operatorname{rank}B=j$.

\begin{rei}\label{rei1}
Let $M=3$, $i=j=1$, $B=[1,0,0]^\top$, and $b_*=[0,0,0]^\top$. Then $\kappa_*=1$. Indeed,
the score is obtained from
the marginal response probabilities before substituting the true parameter values,
\[
p_x(a,b)={\mathbb E}_Z\left[\prod_{\ell=1}^3\sigma\{x_\ell(a_\ell Z-b_\ell)\}\right],\qquad Z\sim N(0,1)
\]
by differentiating with respect to each parameter, evaluating the derivatives at $a=[1,0,0]^\top$ and $b=[0,0,0]^\top$, and dividing by $p_x=1/8$ at the true parameter values.
For $x_1=\pm1$, let $h:={\mathbb E}_Z[\sigma'(Z)]$. Then
\begin{equation}\label{eq802}
{\mathbb E}_Z[\sigma(x_1Z)]=\frac12,\qquad
{\mathbb E}_Z[Z\sigma(x_1Z)]=x_1h,\qquad
{\mathbb E}_Z[Z\sigma'(x_1Z)]=0,\qquad {\mathbb E}_Z[\sigma'(x_1Z)]=h    
\end{equation}
holds. Indeed,
$Z$ and $-Z$ have the same distribution, and $\sigma(z)+\sigma(-z)=1$, so ${\mathbb E}_Z[\sigma(x_1Z)]=1/2$. Moreover,
integration by parts using $\varphi_1'(z)=-z\varphi_1(z)$ gives
\[
{\mathbb E}_Z[Z\sigma(Z)]
=
\int_{\mathbb R}z\sigma(z)\varphi_1(z)dz
=-\int_{\mathbb R}\sigma(z)\varphi_1'(z)\,dz
=\int_{\mathbb R}\sigma'(z)\varphi_1(z)\,dz
=h
\]
as claimed.
Together with ${\mathbb E}_Z[Z]=0$, this gives ${\mathbb E}_Z[Z\sigma(-Z)]={\mathbb E}_Z[Z]-{\mathbb E}_Z[Z\sigma(Z)]=-h$, and hence ${\mathbb E}_Z[Z\sigma(x_1Z)]=x_1h$.
Also, since $\sigma'(\cdot)$ is an even function,
$${\mathbb E}_Z[Z\sigma'(x_1Z)]=0,\qquad
{\mathbb E}_Z[\sigma'(-Z)]={\mathbb E}_Z[\sigma'(Z)]=h
$$
holds.
$$\sigma(0)=\frac{1}{2},\qquad \sigma'(0)=\frac{1}{4}$$
and $p_x=1/8$ at the true parameter values. Thus, by (\ref{eq802}), the score components with respect to
$a=[a_1,a_2,a_3]$, $b=[b_1,b_2,b_3]$
are, respectively,
\[
\begin{aligned}
&8x_1\left(\frac12\right)^2{\mathbb E}_Z[Z\sigma'(x_1Z)]=0,\quad
8x_2\frac14\frac12{\mathbb E}_Z[Z\sigma(x_1Z)]={hx_1x_2},
\quad 8x_3\frac14\frac12{\mathbb E}_Z[Z\sigma(x_1Z)]={hx_1x_3}\\
&\quad -8x_1\left(\frac12\right)^2
{\mathbb E}_Z[\sigma'(x_1Z)]=-2{hx_1},
\quad
-8x_2\frac14\frac12{\mathbb E}_Z[\sigma(x_1Z)]=-\frac{x_2}{2},
\quad -8x_3\frac14\frac12{\mathbb E}_Z[\sigma(x_1Z)]=-\frac{x_3}{2}
\end{aligned}
\]
as calculated above.
When the product of any two distinct components is averaged over the eight response patterns, the signs cancel and the result is zero. By (\ref{eq806}), the Fisher information matrix is
\[
J_*=\frac{1}{8}\sum_{x\in\{-1,1\}^3}s_1(x;v_*)s_1(x;v_*)^\top
=\operatorname{diag}(0,h^2,h^2,4h^2,\frac{1}{4},\frac{1}{4}).
\]
Thus, its kernel has dimension one. Since there are no rotation directions when $j=1$, we obtain $\kappa_*=1$.
\qed
\end{rei}

By Theorem~\ref{prop:irt_rotation_separation} in the preceding section, it suffices to consider directions $[D,0,c]\in\mathcal E_{i,j}$ after excluding rotation directions and directions corresponding to excess ability dimensions. Write the $\ell$th row of $B$ as $a_\ell^\top\in\mathbb R^j$ and the $\ell$th row of $D$ as $d_\ell^\top$, and set
\[
H:=DB^\top+BD^\top,\qquad u_\ell:=H_{\ell\ell}=2a_\ell^\top d_\ell
\]
The $\ell$th row of $A_*$ is $[a_\ell^\top,0]$. We now derive conditions on $u$ for the existence of a direction satisfying both $[D,c]\in\ker J_*$ and $B^\top D=D^\top B$.

\begin{hodai}\label{hodai0}
For the paths $B(s)=B+sD$ and $b(s)=b_*+sc$ in the loading and difficulty parameters,
$$[D,c]\in\ker J_* \quad\Longleftrightarrow\quad \dot p_x := 
\left.\frac{d}{ds}p_j(x\mid B(s),b(s))\right|_{s=0}=0,\quad
x\in \{-1,1\}^M$$
\end{hodai}
\begin{proof}
Let $v(s)\in {\mathbb R}^{Mj+M}$ be the vector formed by listing the entries of $B(s),b(s)$, and write $v(0)=v_*$. Then
$$\dot p_x =\left.\frac{d}{ds}p_j(x\mid v(s))\right|_{s=0}=
\nabla p_j(x\mid v_*)^\top\left.\frac{dv(s)}{ds}\right|_{s=0}
=
\nabla p_j(x\mid v_*)^\top \dot v
$$
where $\displaystyle \dot v:=\left.\frac{dv(s)}{ds}\right|_{s=0}$ is the vector formed by listing the entries of $[D,c]$.
Therefore,
$$\dot v^\top J_* \dot v=
\dot v^\top\left(\sum_{x\in \{-1,1\}^M}\frac{1}{ p_j(x\mid v_*)}
\nabla p_j(x\mid v_*)\nabla p_j(x\mid v_*)^\top\right)\dot v
=
\sum_{x\in \{-1,1\}^M}\frac{(\nabla p_j(x\mid v_*)^\top \dot v)^2}{ p_j(x\mid v_*)}=
\sum_{x\in \{-1,1\}^M}\frac{\dot p_x^2}{ p_j(x\mid v_*)}
$$
Moreover, since $J_*\dot v=0 \Longleftrightarrow \dot v^\top J_* \dot v=0$,
the lemma follows.
\end{proof}

Write $z_\ell$ for the inner product of each item's discrimination vector and the ability vector, and define the conditional response probability by
\[
F_x(z):=\prod_{\ell=1}^M\sigma\{x_\ell(z_\ell-b_{\ell*})\},\qquad z=[z_1,\ldots,z_M]^\top
\]
For $\Theta\sim N(0,I_j)$, we have $p_j(x\mid B,b_*)={\mathbb E}_\Theta[F_x(B\Theta)]$. In what follows, $\partial_\ell$ denotes partial differentiation with respect to $z_\ell$, and all expectations are evaluated at the true parameter values.

Let $p_\ell$ denote the probability of a correct response to item $\ell$, and let $p_{\ell r}$ denote the probability of correct responses to both items $\ell,r$. Set
\[
\sigma_\ell:=\sigma(a_\ell^\top\Theta-b_{\ell*}),\qquad
p_\ell={\mathbb E}_\Theta[\sigma_\ell],\qquad p_{\ell r}={\mathbb E}_\Theta[\sigma_\ell\sigma_r]
\]
Here, $\sigma_\ell',\sigma_\ell''$ denote the first and second derivatives of $\sigma$ with respect to its argument, evaluated at $a_\ell^\top\Theta-b_{\ell*}$.

Define
\begin{equation}\label{eq:kg-single-coefficients}
\alpha_\ell:={\mathbb E}_\Theta[\sigma_\ell']>0,\qquad
\eta_\ell:=\frac{{\mathbb E}_\Theta[\sigma_\ell'']}{2\alpha_\ell}
\end{equation}
\begin{equation}\label{eq:kg-pair-coefficients}
\beta_{\ell r}:={\mathbb E}_\Theta[\sigma_\ell'\sigma_r']>0,\qquad
\rho_{\ell r}:=-\frac{\frac12{\mathbb E}_\Theta[\sigma_\ell''\sigma_r]-\eta_\ell{\mathbb E}_\Theta[\sigma_\ell'\sigma_r]}{\beta_{\ell r}}\quad(\ell\ne r)
\end{equation}
The positivity of each denominator follows from $\sigma'(s)>0$. Also,
for $u=[u_1,\ldots,u_M]^\top$, define the symmetric matrix $H(u)$ by
$$
H_{\ell\ell}(u):=u_\ell,\qquad H_{\ell r}(u):=\rho_{\ell r}u_\ell+\rho_{r\ell}u_r\quad(\ell\ne r)
$$
and, for $x\in \{-1,1\}^M, \ell=1,\ldots,M$, define
\begin{align}
h_{x,\ell}:={}&\frac12{\mathbb E}_\Theta[\partial_\ell^2F_x(B\Theta)]-\eta_\ell{\mathbb E}_\Theta[\partial_\ell F_x(B\Theta)]+\sum_{r\ne\ell}\rho_{\ell r}{\mathbb E}_\Theta[\partial_\ell\partial_rF_x(B\Theta)]\label{eq:kg-h-coefficients}
\end{align}
Then the following proposition holds.

\begin{prop}\label{prop11}
Let $u=(2a_1^\top d_1,\ldots,2a_M^\top d_M)^\top$. If $c_\ell=\eta_\ell u_\ell$ and $H=H(u)$, then
$$
\dot p_x=\sum_{\ell=1}^M h_{x,\ell}u_\ell$$
holds. Moreover, if $[D,c]\in \ker J_*$, then
$c_\ell=\eta_\ell u_\ell$,\quad $\ell=1,\ldots,M$,\quad $H=H(u)$ holds.
\end{prop}
\begin{proof} 
The probabilities $p_\ell, p_{\ell r}$ are obtained from $p_x$ by restricting to the response patterns satisfying $x_\ell=1$ and $x_\ell=x_r=1$, respectively, and summing the corresponding $p_x$. Thus, if we define
\[
\begin{aligned}
p_\ell(s)
&:={\mathbb E}_\Theta\left[
\sigma\{(a_\ell+sd_\ell)^\top\Theta-b_{\ell*}-sc_\ell\}
\right]\\
p_{\ell r}(s)
&:={\mathbb E}_\Theta\bigl[
\sigma\{(a_\ell+sd_\ell)^\top\Theta-b_{\ell*}-sc_\ell\}\cdot\sigma\{(a_r+sd_r)^\top\Theta-b_{r*}-sc_r\}
\bigr].
\end{aligned}
\]
then
\begin{equation}\label{eq711}
\dot p_x=0,\quad x\in\{-1,1\}^M\quad \Longrightarrow
\quad
\frac{dp_\ell(s)}{ds}|_{s=0}=\frac{dp_{\ell r}(s)}{ds}|_{s=0}=0,\quad
\ell=1,\ldots,M,\quad r\not=\ell
\end{equation}
holds. First,
\begin{align}
\frac{d}{ds}p_\ell(s)|_{s=0}
=&{\mathbb E}_\Theta[(d_\ell^\top\Theta-c_\ell)\sigma_\ell']
=\sum_{k=1}^j d_{\ell k}
{\mathbb E}_\Theta[\Theta_k\sigma_\ell']
-c_\ell{\mathbb E}_\Theta[\sigma_\ell']\nonumber\\
=&\sum_{k=1}^j d_{\ell k}a_{\ell k}
{\mathbb E}_\Theta[\sigma_\ell'']
-c_\ell{\mathbb E}_\Theta[\sigma_\ell']
=d_\ell^\top a_\ell
{\mathbb E}_\Theta[\sigma_\ell'']
-c_\ell{\mathbb E}_\Theta[\sigma_\ell']\label{eq991}
\end{align}
The third equality follows by the same argument as in Example~\ref{rei1} (integration by parts for the normal distribution).
Here, $d_{\ell k}$ denotes the $k$th component of $d_\ell$.
Therefore, for each $\ell=1,\ldots,M$,
\begin{equation}\label{eq712}
\frac{d}{ds}p_\ell(s)|_{s=0}=0 \quad \Longleftrightarrow\quad
c_\ell=\eta_\ell u_\ell
\end{equation}
holds.

Next, differentiating
$$p_j(x\mid B(s),b(s))={\mathbb E}_\Theta[F_x(B\Theta+s(D\Theta-c))]$$
with respect to $s$ and setting $s=0$ gives,
with $H=DB^\top+BD^\top=(H_{\ell r})$,
\begin{align}
\dot p_x={}&\sum_{\ell=1}^M{\mathbb E}_\Theta[(d_\ell^\top\Theta-c_\ell)\partial_\ell F_x(B\Theta)]\nonumber \\
={}&\sum_{\ell=1}^M\left\{\frac{H_{\ell\ell}}2{\mathbb E}_\Theta[\partial_\ell^2F_x(B\Theta)]-c_\ell{\mathbb E}_\Theta[\partial_\ell F_x(B\Theta)]\right\}+\sum_{\ell<r}H_{\ell r}{\mathbb E}_\Theta[\partial_\ell\partial_rF_x(B\Theta)]\label{eq:kg-probability-derivative}
\end{align}
Indeed, integration by parts for the normal distribution, ${\mathbb E}_\Theta[\Theta_k g(\Theta)]={\mathbb E}_\Theta[(\partial g/\partial\theta_k)(\Theta)]$, yields
\[
{\mathbb E}_\Theta[\Theta\partial_\ell F_x(B\Theta)]
=\sum_{r=1}^M a_r{\mathbb E}_\Theta[\partial_\ell\partial_rF_x(B\Theta)]
\]
For the terms with $r=\ell$, use $d_\ell^\top a_\ell=H_{\ell\ell}/2$. For the terms with $\ell\ne r$, combine the two terms with the indices interchanged and use $d_\ell^\top a_r+d_r^\top a_\ell=H_{\ell r}$. This gives \eqref{eq:kg-probability-derivative}. Moreover, $\sigma,\sigma',\sigma''$ are bounded, and ${\mathbb E}_\Theta\|\Theta\|<\infty$, so the interchange of differentiation and expectation and the integration by parts are justified.

Furthermore, since $c_\ell=\eta_\ell u_\ell$, $c_r=\eta_r u_r$,
\begin{align}
&\frac{d}{ds}p_{\ell r}(s)|_{s=0}
={}{\mathbb E}_\Theta[
(d_\ell^\top\Theta-c_\ell)\sigma_\ell'\sigma_r]+{\mathbb E}_\Theta[
(d_r^\top\Theta-c_r)\sigma_\ell\sigma_r']\nonumber\\
={}&d_\ell^\top a_\ell
{\mathbb E}_\Theta[\sigma_\ell''\sigma_r]
+d_r^\top a_r
{\mathbb E}_\Theta[\sigma_\ell\sigma_r'']
+(d_\ell^\top a_r+d_r^\top a_\ell)
{\mathbb E}_\Theta[\sigma_\ell'\sigma_r']
-c_\ell{\mathbb E}_\Theta[\sigma_\ell'\sigma_r]
-c_r{\mathbb E}_\Theta[\sigma_\ell\sigma_r']\nonumber\\
=&
\frac{1}{2}u_\ell{\mathbb E}[\sigma_{\ell}''\sigma_r]
+
\frac{1}{2}u_r{\mathbb E}[\sigma_{\ell}\sigma_r'']
+H_{\ell r}{\mathbb E}_\Theta[\sigma_\ell'\sigma_r']
-\eta_\ell u_\ell{\mathbb E}[\sigma_\ell'\sigma_r]
-\eta_ru_r{\mathbb E}[\sigma_\ell\sigma_r']
\label{eq992}\\
=&
\beta_{\ell r}\{H_{\ell r}-\rho_{\ell r}u_\ell-\rho_{r\ell}u_r\}\nonumber
\end{align}
Thus, by Lemma~\ref{hodai0} and \eqref{eq711}\eqref{eq712}\eqref{eq992},
$$
[D,c]\in \ker J_*
\quad \Longrightarrow \quad
\dot p_x=0,\quad x\in\{-1,1\}^M
\quad \Longrightarrow
\quad
c_\ell=\eta_\ell u_\ell,\quad \ell=1,\ldots,M, \quad H=H(u)$$
holds.
Finally, assuming $c_\ell=\eta_\ell u_\ell$ and $H=H(u)$ and substituting them into \eqref{eq:kg-probability-derivative}
gives the proposition.
\end{proof}

\begin{hodai}\label{lem:kg-gram-image}
Let $N=[n_1,\ldots,n_{M-j}]\in\mathbb R^{M\times(M-j)}$ be any basis matrix for $\ker B^\top$. For a symmetric matrix $H\in\mathbb R^{M\times M}$,
\[
H=DB^\top+BD^\top\ \text{for some }D\in\mathbb R^{M\times j}
\quad\Longleftrightarrow\quad N^\top HN=0
\]
holds. Moreover, the matrices $D$ satisfying $DB^\top+BD^\top=0$ are precisely those of the form $D=B\Omega$, $\Omega^\top=-\Omega$.
\end{hodai}
Proof: See Appendix~B.

Since every entry of $H(u)$ is a linear combination of $u_1,\ldots,u_M$,
\begin{align}
N^\top H(u)N&=0,\label{eq:kg-L-definition1}\\
\sum_{\ell=1}^M h_{x,\ell}u_\ell&=0,\qquad x\in\{-1,1\}^M\label{eq:kg-L-definition2}
\end{align}
form a system of linear equations in $u_1,\ldots,u_M$. Write (\ref{eq:kg-L-definition1}) for the upper triangular entries, and let $L_j=L_j(B,b_*)$ be the matrix with $M$ columns obtained by collecting the coefficients of these equations. Thus, the two conditions are equivalent to
\begin{equation}\label{eq:kg-L-equations}
L_ju=0
\end{equation}
When $j=M$, (\ref{eq:kg-L-definition1}) is absent.

We now show that
\begin{equation}\label{map1}
\Phi: \mathcal E_{i,j}\ni [D,0,c] \mapsto u=(2a_1^\top d_1,\ldots,2a_M^\top d_M)^\top\in \ker L_j
\end{equation}
is a bijective linear map.

\begin{teiri}[Formula for $\kappa$]\label{thm:kg-kappa-formula}
For any $1\leq j\leq M$, any $B\in\mathbb R^{M\times j}$ satisfying $\operatorname{rank}B=j$, and any $b_*\in\mathbb R^M$, the matrix $L_j(B,b_*)$ defined above satisfies
\begin{equation}\label{eq:kg-kappa-formula}
\boxed{\kappa_*=M-\operatorname{rank}L_j(B,b_*)}
\end{equation}
Consequently,
\begin{equation}\label{eq:kg-fisher-rank}
\operatorname{rank}J_*=Mj-\frac{j(j-1)}2+\operatorname{rank}L_j(B,b_*).
\end{equation}
In particular, the kernel in the $j$-dimensional model contains no directions other than rotation directions if and only if $\operatorname{rank}L_j=M$. These conclusions do not depend on the choice of basis for $\ker B^\top$.
\end{teiri}

\begin{proof}
First, take $[D,0,c]\in\mathcal E_{i,j}$. By definition, $[D,c]\in\ker J_*$, so Lemma~\ref{hodai0} gives $\dot p_x=0$, $x\in \{-1,1\}^M$. Also, Proposition~\ref{prop11} gives
$$c_\ell=\eta_\ell u_\ell,\qquad DB^\top+BD^\top=H(u)$$
Furthermore, Lemma~\ref{lem:kg-gram-image} gives (\ref{eq:kg-L-definition1}), and Proposition~\ref{prop11} gives (\ref{eq:kg-L-definition2}). Thus, $u\in\ker L_j$, and $\Phi$ is well defined.

Next, we prove surjectivity. Take $u\in\ker L_j$. By (\ref{eq:kg-L-definition1}) and Lemma~\ref{lem:kg-gram-image}, there exists $D$ such that $DB^\top+BD^\top=H(u)$. Set $c_\ell=\eta_\ell u_\ell$. Then Proposition~\ref{prop11} and (\ref{eq:kg-L-definition2}) give $\dot p_x=0$,\quad $x\in \{-1,1\}^M$. Hence, Lemma~\ref{hodai0} implies $[D,c]\in\ker J_*$.
Using the decomposition $D=B\Omega+D_\perp$ from the preceding section to remove the rotation component gives
$$[D_\perp,c]\in\ker J_*,\qquad B^\top D_\perp=D_\perp^\top B$$
so $[D_\perp,0,c]\in\mathcal E_{i,j}$. Moreover, since $\Omega^\top=-\Omega$,
$$D_\perp B^\top+BD_\perp^\top
=H(u)-B(\Omega+\Omega^\top)B^\top=H(u)$$
holds.
Since the $(\ell,\ell)$ entry of $H(u)$ is $u_\ell$,
we have $\Phi([D_\perp,0,c])=u$. Thus, $\Phi$ is surjective.

Finally, we prove injectivity. If $\Phi([D,0,c])=0$, then $u=0$, so $c=0$ and $H=H(0)=0$. By Lemma~\ref{lem:kg-gram-image}, $D=B\Omega$. Therefore, $[D,0,0]$ belongs to both $\mathcal E_{i,j}$ and $\mathcal R_{i,j}$. By the orthogonal decomposition in the preceding section, their intersection consists only of zero. Hence, $D=0$, and $\Phi$ is also injective.

Thus, $\Phi$ is bijective. Consequently,
$$\kappa_*=\dim\mathcal E_{i,j}=\dim\ker L_j
=M-\operatorname{rank}L_j$$
is obtained. That is, $\kappa_*$ is the number of free parameters in the general solution of the system of linear equations in the $M$ unknowns $u_1,\ldots,u_M$. The rank formula \eqref{eq:kg-fisher-rank} follows from \eqref{eq:irt_exact_rank_with_residual} in the preceding section and Lemma~\ref{hodai1}.
Finally, replacing $N$ by another basis matrix $NR$ does not change the condition $N^\top H(u)N=0$, because $R$ is invertible. Thus, $\ker L_j$ and its dimension are also unchanged.
\end{proof}



\begin{kei}[Disjoint column supports]\label{cor:kg-disjoint}
Assume that $\operatorname{rank}B=j$, and let
\[
S_r:=\{\ell:B_{\ell r}\ne0\},\qquad n_r:=|S_r|\quad(r=1,\ldots,j)
\]
If $S_1,\ldots,S_j$ are pairwise disjoint, then, for any $b_*$,
\[
q:=\#\{r:n_r=1\text{ or }n_r=2\},\qquad
\boxed{\kappa_*=q},\qquad \operatorname{rank}L_j=M-q
\]
holds.
\end{kei}

\begin{proof}
By assumption, the nonzero entries in row $\ell$ occur in at most one column $r$. We denote this relation by
$\ell\in S_r$ and write $\gamma_\ell:=B_{\ell r}\ne0$.

The proof consists of establishing
\begin{equation}\label{eq43}
w_{\ell m}:=\gamma_m^2-2\gamma_\ell\gamma_m\rho_{\ell m}>0, \quad \ell,m\in S_r,\quad \ell\ne m,\quad r=1,\ldots,j
\end{equation}
and, for each $[u_1,\ldots,u_M]$, showing that the conditions
\begin{equation}\label{eq:kg-disjoint-linear}
u_\ell=0\quad\left(\ell\notin\bigcup_rS_r\right),\qquad
w_{\ell m}u_\ell+w_{m\ell}u_m=0
\quad(\ell<m,\ \ell,m\in S_r).
\end{equation}
are equivalent to the first condition:
$$\eqref{eq:kg-L-definition1} \Longleftrightarrow \eqref{eq:kg-disjoint-linear}$$
and that \eqref{eq:kg-L-definition2} also holds under these conditions.

These assertions imply the following conclusions.
If \eqref{eq:kg-L-definition1}
holds, then, in \eqref{eq:kg-disjoint-linear},
the case $n_r=1$ has one variable and no constraint, whereas the case $n_r=2$ has two variables and one equation. Thus, the solution space is one-dimensional in both cases.
When $n_r\ge3$, for distinct $\ell,m,k\in S_r$, substituting
\[
u_m=-\frac{w_{\ell m}}{w_{m\ell}}u_\ell,\qquad
u_k=-\frac{w_{\ell k}}{w_{k\ell}}u_\ell
\]
into the equation for $m,k$ gives
\[
0=
w_{mk}u_m+w_{km}u_k=
-\left(
\frac{w_{mk}w_{\ell m}}{w_{m\ell}}
+\frac{w_{km}w_{\ell k}}{w_{k\ell}}
\right)u_\ell
\]
Since the expression in parentheses is positive, $u_\ell=0$, and $u_m,u_k$ and the remaining variables are also zero. Moreover, \eqref{eq:kg-L-definition1}$\Longrightarrow $\eqref{eq:kg-L-definition2}, so $L_ju=0$ is
equivalent to \eqref{eq:kg-L-definition1}.

First, we prove \eqref{eq43}.
For items in the same $S_r$ indexed by $\ell,m$, let
$\sigma_\ell=\sigma(\gamma_\ell Z-b_{\ell*})$ and $\sigma_m=\sigma(\gamma_m Z-b_{m*})$.
Integration by parts for the normal distribution, the product rule, and the definitions of $\eta_\ell$, $\alpha_\ell$, and $\beta_{\ell m}$ in \eqref{eq:kg-single-coefficients} and \eqref{eq:kg-pair-coefficients} give, for $Z\sim N(0,1)$,
\begin{align}\label{eq13}
{\mathbb E}_Z[(Z-2\gamma_\ell\eta_\ell)\sigma_\ell']
={\mathbb E}_Z[Z\sigma_\ell']
-2\gamma_\ell\eta_\ell{\mathbb E}_Z[\sigma_\ell']
=\gamma_\ell(2\alpha_\ell\eta_\ell)-2\gamma_\ell\eta_\ell\alpha_\ell=0.
\end{align}
\begin{align}\label{eq114}
{\mathbb E}_Z[Z\sigma_\ell'\sigma_m]
&=\gamma_\ell{\mathbb E}_Z[\sigma_\ell''\sigma_m]
+\gamma_m{\mathbb E}_Z[\sigma_\ell'\sigma_m']=\gamma_\ell{\mathbb E}_Z[\sigma_\ell''\sigma_m]
+\gamma_m\beta_{\ell m}
\end{align}
Here, we have used
\[
\frac{d}{dz}\sigma'(\gamma_\ell z-b_{\ell*})
=\gamma_\ell\sigma''(\gamma_\ell z-b_{\ell*}),\qquad
{\mathbb E}_Z[Z\sigma_\ell']
=\gamma_\ell{\mathbb E}_Z[\sigma_\ell'']
\]
Hence,
by
the definition of $\rho_{\ell m}$ in \eqref{eq:kg-pair-coefficients} and \eqref{eq114},
\begin{align}
w_{\ell m}
&=\gamma_m^2+\frac{2\gamma_\ell\gamma_m}{\beta_{\ell m}}
\left\{\frac12{\mathbb E}_Z[\sigma_\ell''\sigma_m]
-\eta_\ell{\mathbb E}_Z[\sigma_\ell'\sigma_m]\right\}
=\frac{\gamma_m}{\beta_{\ell m}}
\left\{\gamma_m\beta_{\ell m}
+\gamma_\ell{\mathbb E}_Z[\sigma_\ell''\sigma_m]
-2\gamma_\ell\eta_\ell{\mathbb E}_Z[\sigma_\ell'\sigma_m]\right\}\nonumber\\
&=\frac{\gamma_m}{\beta_{\ell m}}
\left\{{\mathbb E}_Z[Z\sigma_\ell'\sigma_m]
-2\gamma_\ell\eta_\ell{\mathbb E}_Z[\sigma_\ell'\sigma_m]\right\}=\frac{\gamma_m}{\beta_{\ell m}}
{\mathbb E}_Z[(Z-2\gamma_\ell\eta_\ell)\sigma_\ell'\sigma_m]\nonumber\\
&=
\frac1{\beta_{\ell m}}
{\mathbb E}_Z\bigl[
\sigma_\ell'\gamma_m(Z-2\gamma_\ell\eta_\ell)
\{\sigma(\gamma_mZ-b_{m*})-\sigma(2\gamma_\ell\gamma_m\eta_\ell-b_{m*})\}\bigr]>0
\end{align}
holds.
The last equality follows because
$\sigma(2\gamma_\ell\gamma_m\eta_\ell-b_{m*})$ is constant with respect to $Z$ and, by \eqref{eq13},
\[
\begin{aligned}
&{\mathbb E}_Z\bigl[
(Z-2\gamma_\ell\eta_\ell)\sigma_\ell'
\sigma(2\gamma_\ell\gamma_m\eta_\ell-b_{m*})\bigr]
 =
\sigma(2\gamma_\ell\gamma_m\eta_\ell-b_{m*})
{\mathbb E}_Z[(Z-2\gamma_\ell\eta_\ell)\sigma_\ell']=0.
\end{aligned}
\]
holds. The last inequality follows from the following inequalities.

\[
\beta_{\ell m}>0, \qquad \sigma_\ell'>0,\qquad
\gamma_m(z-2\gamma_\ell\eta_\ell)
\{\sigma(\gamma_mz-b_{m*})-\sigma(2\gamma_\ell\gamma_m\eta_\ell-b_{m*})\}>0
\qquad(z\ne2\gamma_\ell\eta_\ell)
\]

Next, we prove
$\eqref{eq:kg-L-definition1} \Longrightarrow \eqref{eq:kg-disjoint-linear}$.
Take $u$ satisfying \eqref{eq:kg-L-definition1}. By Lemma \ref{lem:kg-gram-image},
$DB^\top+BD^\top=H(u)$ holds for some $D$.
If $\ell\notin\bigcup_rS_r$, then $a_\ell=0$, so $u_\ell=0$.
For $\ell,m\in S_r$,
$u_\ell=2\gamma_\ell D_{\ell r}$ and
$u_m=2\gamma_m D_{m r}$
give
\[
H_{\ell m}(u)
=\gamma_mD_{\ell r}+\gamma_\ell D_{mr}
=\frac{\gamma_m}{2\gamma_\ell}u_\ell
+\frac{\gamma_\ell}{2\gamma_m}u_m.
\]
Together with $H_{\ell m}(u)=\rho_{\ell m}u_\ell+\rho_{m\ell}u_m$, this yields
$$(\frac{\gamma_m}{2\gamma_\ell}-\rho_{\ell m})u_\ell+(\frac{\gamma_\ell}{2\gamma_m}-\rho_{m\ell})u_m=0$$
Multiplying both sides by $2\gamma_\ell\gamma_m$ and using the definitions of $w_{\ell m},w_{m\ell}$
gives \eqref{eq:kg-disjoint-linear}.

We now prove
$\eqref{eq:kg-L-definition1} \Longleftarrow \eqref{eq:kg-disjoint-linear}$.
Take $u$ satisfying \eqref{eq:kg-disjoint-linear}, and let
\begin{equation}\label{eq999}
D_{\ell r}:=
\begin{cases}
u_\ell/(2\gamma_\ell),&\ell\in S_r,\\
0,&\ell\notin S_r,
\end{cases}
\qquad c_\ell:=\eta_\ell u_\ell
\end{equation}
By Lemma \ref{lem:kg-gram-image},
for this $D$, it suffices to show that all entries of $DB^\top+BD^\top$ and $H(u)$ agree.
First, if $\ell\in S_r$, the definition of $D$ gives
\[
(DB^\top+BD^\top)_{\ell\ell}
=2\gamma_\ell D_{\ell r}
=2\gamma_\ell\frac{u_\ell}{2\gamma_\ell}
=u_\ell=H_{\ell\ell}(u).
\]
If $\ell\notin\bigcup_rS_r$, then $B,D$ both have a zero row $\ell$, and $u_\ell=0$ by \eqref{eq:kg-disjoint-linear}. Thus, the diagonal entries also agree in this case.
Next, for items in the same $S_r$ with $\ell\ne m$, \eqref{eq:kg-disjoint-linear} gives
\[
w_{\ell m}u_\ell+w_{m\ell}u_m=0
\]
Substituting the definitions of $w_{\ell m},w_{m\ell}$ and dividing by $2\gamma_\ell\gamma_m$, we obtain
\[
\left(\frac{\gamma_m}{2\gamma_\ell}-\rho_{\ell m}\right)u_\ell
+\left(\frac{\gamma_\ell}{2\gamma_m}-\rho_{m\ell}\right)u_m=0.
\]
Therefore,
\[
\begin{aligned}
(DB^\top+BD^\top)_{\ell m}
=\gamma_mD_{\ell r}+\gamma_\ell D_{mr}
=\frac{\gamma_m}{2\gamma_\ell}u_\ell+\frac{\gamma_\ell}{2\gamma_m}u_m
=\rho_{\ell m}u_\ell+\rho_{m\ell}u_m
=H_{\ell m}(u).
\end{aligned}
\]
Furthermore, if $\ell\in S_r$, $m\in S_t$, and $r\ne t$, $B,D$ have zeros in row $\ell$ outside column $r$ and in row $m$ outside column $t$. Hence,
\[
(DB^\top+BD^\top)_{\ell m}
=\sum_{k=1}^j(D_{\ell k}B_{mk}+B_{\ell k}D_{mk})=0.
\]
This sum is also zero when one of the items has a zero row. Moreover, for items in distinct sets $S_r,S_t$, $\sigma_\ell$ and its derivatives are functions of $\Theta_r$ alone, and $\sigma_m$ is a function of $\Theta_t$ alone. The independence of $\Theta_r,\Theta_t$ therefore gives
\[
\frac12{\mathbb E}_\Theta[\sigma_\ell''\sigma_m]
-\eta_\ell{\mathbb E}_\Theta[\sigma_\ell'\sigma_m]
=\left(\frac12{\mathbb E}_\Theta[\sigma_\ell'']
-\eta_\ell{\mathbb E}_\Theta[\sigma_\ell']\right)
{\mathbb E}_\Theta[\sigma_m]=0.
\]
The last equality follows from the definition of $\eta_\ell$. The same first equality holds when one of the items has a zero row, since $\sigma$ and its derivatives for that item are constant. Thus, $\rho_{\ell m}=0$. Interchanging $\ell,m$ also gives $\rho_{m\ell}=0$, so, for these item pairs,
\[
H_{\ell m}(u)=\rho_{\ell m}u_\ell+\rho_{m\ell}u_m=0.
\]
Since all entries agree, this $D$ satisfies $DB^\top+BD^\top=H(u)$.
Thus, \eqref{eq:kg-L-definition1} has been proved.

Finally, we prove $\eqref{eq:kg-L-definition1}\Longrightarrow
\eqref{eq:kg-L-definition2}$.
In \eqref{eq991}, obtained in the proof of Proposition \ref{prop11}, replace $2a_\ell^\top d_\ell$
by $u_\ell$ and
substitute $c_\ell=\eta_\ell u_\ell$. By the definitions of $\alpha_\ell,\eta_\ell$ in \eqref{eq:kg-single-coefficients},
\[
\begin{aligned}
\left.\frac{d}{ds}p_\ell(s)\right|_{s=0}
=\frac{u_\ell}{2}{\mathbb E}_\Theta[\sigma_\ell'']
-c_\ell{\mathbb E}_\Theta[\sigma_\ell']
=\frac{u_\ell}{2}(2\alpha_\ell\eta_\ell)
-(\eta_\ell u_\ell)\alpha_\ell=0.
\end{aligned}
\]
For items in the same $S_r$ with $\ell\ne m$, replacing $2a_\ell^\top d_\ell$, $2a_m^\top d_m$ by $u_\ell,u_m$ in \eqref{eq992}, obtained in the proof of Proposition \ref{prop11}, gives
\[
\begin{aligned}
\left.\frac{d}{ds}p_{\ell m}(s)\right|_{s=0}
={}\frac{u_\ell}{2}{\mathbb E}_\Theta[\sigma_\ell''\sigma_m]
-c_\ell{\mathbb E}_\Theta[\sigma_\ell'\sigma_m]+\frac{u_m}{2}{\mathbb E}_\Theta[\sigma_\ell\sigma_m'']
-c_m{\mathbb E}_\Theta[\sigma_\ell\sigma_m']
+(DB^\top+BD^\top)_{\ell m}\beta_{\ell m}.
\end{aligned}
\]
Substituting $c_\ell=\eta_\ell u_\ell$ and $c_m=\eta_m u_m$, the definitions of $\rho_{\ell m},\rho_{m\ell}$ in \eqref{eq:kg-pair-coefficients} show that the first two terms equal $-\beta_{\ell m}\rho_{\ell m}u_\ell$, and the next two terms equal $-\beta_{\ell m}\rho_{m\ell}u_m$. Also, the definition of $D$ gives
\[
(DB^\top+BD^\top)_{\ell m}
=\frac{\gamma_m}{2\gamma_\ell}u_\ell
+\frac{\gamma_\ell}{2\gamma_m}u_m.
\]
Therefore,
\[
\begin{aligned}
\left.\frac{d}{ds}p_{\ell m}(s)\right|_{s=0}
=\beta_{\ell m}\left\{
\left(\frac{\gamma_m}{2\gamma_\ell}-\rho_{\ell m}\right)u_\ell
+\left(\frac{\gamma_\ell}{2\gamma_m}-\rho_{m\ell}\right)u_m
\right\}
=\frac{\beta_{\ell m}}{2\gamma_\ell\gamma_m}
(w_{\ell m}u_\ell+w_{m\ell}u_m)=0.
\end{aligned}
\]
The second equality follows from the definitions of $w_{\ell m},w_{m\ell}$. The last equality follows because the chosen $u$ satisfies \eqref{eq:kg-disjoint-linear}.
When $n_r=1$, the two response probabilities are $p_\ell(s),1-p_\ell(s)$; when $n_r=2$, the four response probabilities are
\[
p_{\ell m}(s),\quad p_\ell(s)-p_{\ell m}(s),\quad
p_m(s)-p_{\ell m}(s),\quad
1-p_\ell(s)-p_m(s)+p_{\ell m}(s)
\]
Thus, all their derivatives are zero. When $n_r\ge3$, for $\ell\in S_r$, each $u_\ell$ has already been shown to be zero. By \eqref{eq999} and $c_\ell=\eta_\ell u_\ell$,
$D,c$ are zero for all items in this set, and all their response probabilities are independent of $s$.
For $S_r$, write $p^{(r)}_{x_{S_r}}(s)$ for the within-group response probability. For $B(s)=B+sD$ and $b(s)=b_*+sc$, the definition of $D$ shows that each item $\ell\in S_r$ depends only on $\Theta_r$. Also, $u_\ell=0$ for a zero row, so $c_\ell=0$. By the independence of $\Theta_1,\ldots,\Theta_j$,
\[
p_x(s)=\prod_{r=1}^j p^{(r)}_{x_{S_r}}(s)
\prod_{\ell\notin\bigcup_rS_r}\sigma(-x_\ell b_{\ell*}).
\]
For $n_r=1,2$, the factors $p_{xS_r}^{(r)}$ are the only ones that can depend on $s$. Thus,
since the derivative of each factor at $s=0$ is zero, differentiating the product gives
\[
\left.\frac{d}{ds}p_x(s)\right|_{s=0}=0\,\qquad x\in\{-1,1\}^M
\]
Substituting the previously established identities $DB^\top+BD^\top=H(u)$ and $c_\ell=\eta_\ell u_\ell$ into \eqref{eq:kg-probability-derivative}, we obtain
\[
\sum_{\ell=1}^M h_{x,\ell}u_\ell=\left.\frac{d}{ds}p_x(s)\right|_{s=0}=0.
\]
Therefore, \eqref{eq:kg-L-definition2} also holds.

It follows that $\ker L_j$ is the solution space of \eqref{eq:kg-disjoint-linear}, and its dimension is $q$.
By Theorem \ref{thm:kg-kappa-formula}, we obtain the following conclusion.
\[
\kappa_*=\dim\ker L_j=q,\qquad \operatorname{rank}L_j=M-q
\]

\end{proof}

\begin{rei}[Groups of three and two items]\label{rei2}
Let $M=5$, $i=j=2$, $b_*=0$, and
\[
B=\begin{bmatrix}1&0\\1&0\\1&0\\0&1\\0&1\end{bmatrix}
\]
Since $S_1=\{1,2,3\}$ and $S_2=\{4,5\}$, we have $n_1=3$, $n_2=2$, and $q=1$.
By Corollary \ref{cor:kg-disjoint}, we obtain the following values.
\[
\boxed{\kappa_*=1},\qquad \operatorname{rank}L_2=5-1=4
\]

For Example \ref{rei1}, $S_1=\{1\}$ and $n_1=q=1$, so the conclusion $\kappa_*=1$ also follows immediately from Corollary \ref{cor:kg-disjoint}.
\qed
\end{rei}

\section{Conclusions and Future Work}
We analyzed the Fisher information matrix of the marginal response distribution obtained by integrating out standard normally distributed abilities in the multidimensional 2PL model, and identified the parameter directions that cause degeneracy and their dimensions. In particular, we distinguished degeneracy due to rotational invariance from that due to an over-specified latent dimension, and provided a method for determining from the item parameters whether degeneracy remains after these directions have been removed.

The first result is the orthogonal decomposition of the entire kernel for general $1\leq j\leq i\leq M$,
\[
\ker I_*=\mathcal R_{i,j}\oplus\mathcal N_{i,j}\oplus\mathcal E_{i,j}
\]
Here, $\mathcal R_{i,j}$ consists of rotation directions, $\mathcal N_{i,j}$ consists of directions associated with unused ability coordinates that cannot be represented by rotations, and $\mathcal E_{i,j}$ consists of directions within the $j$-dimensional model that are not due to rotations. The space of rotation directions has dimension $j(j-1)/2+j(i-j)$, while the space of directions associated with unused ability coordinates that cannot be represented by rotations has dimension $(M-j)(i-j)$. This allows the effects of rotations and an over-specified latent dimension to be counted without overlap. Writing $\kappa_*:=\dim\mathcal E_{i,j}$, we obtained
\[
\dim\ker I_*=\frac{j(j-1)}2+M(i-j)+\kappa_*,\qquad
\operatorname{rank}I_*=Mj+M-\frac{j(j-1)}2-\kappa_*
\]
Thus, the rank of the Fisher information matrix is determined by the number of items $M$, the rank $j$ of the true loading matrix, and the degeneracy that remains within the $j$-dimensional model. Increasing the candidate dimension increases the dimension of the kernel, but does not change the rank when the true item parameters $B,b_*$ are fixed.

The second result is the formula for the dimension of this remaining degeneracy,
\[
\kappa_*=M-\operatorname{rank}L_j(B,b_*)
\]
The matrix $L_j(B,b_*)$ has $M$ columns and is constructed from the true discrimination and difficulty parameters. The calculation of $\kappa_*$ is thereby reduced to finding the dimension of the solution space of a system of linear equations in $M$ unknowns. Consequently, the necessary and sufficient condition for no degeneracy to remain after rotation directions and directions associated with unused ability coordinates have been removed is $\operatorname{rank}L_j(B,b_*)=M$. This condition determines whether degeneracy remains after the candidate dimension has been set equal to the rank of the loading matrix and degeneracy due to rotations has been removed.

Furthermore, when the nonzero positions in the columns of the loading matrix are disjoint, we showed that $\kappa_*$ equals the number of columns with one or two nonzero entries. This result holds for arbitrary difficulty parameters. Under this structure, each ability direction associated with only one or two items gives rise to one additional direction of degeneracy. Therefore, in the $j$-dimensional model with this structure, at least three items must be associated with every ability direction for there to be no degeneracy other than that due to rotations, and this condition is also sufficient. Thus, even without an over-specified latent dimension, degeneracy can arise from the correspondence between ability directions and items, and its dimension can be determined.

A topic for future work is to examine the second- and higher-order changes in the marginal response probabilities along the kernel directions identified here. Among directions along which the first derivatives vanish, some produce changes in the distribution at higher orders, whereas others leave the distribution unchanged as the parameters vary, as in Example~\ref{rei1}. These cases need to be distinguished in analyzing the likelihood near singular points, deriving the limiting distributions of likelihood ratio statistics and the learning coefficients of the marginal model, and applying these results to latent dimension selection. The results of this paper identify all directions of first-order degeneracy that need to be examined in such analyses and give the dimension of the degeneracy associated with each cause.

\bibliography{ref}
\appendix

\section*{Appendix A: Proof of Lemma~\ref{hodai1}}

\begin{proof}
Set $r=i-j$ and partition the ability vector as $\theta=(\zeta^\top,\xi^\top)^\top$, where $\zeta\sim N(0,I_j)$ and $\xi\sim N(0,I_r)$ are independent. Differentiating Eq.~\eqref{eq:irt_marg_submodel_identity} with respect to $B,b$ shows that the scores for the first $j$ columns and the difficulties coincide with $s_j(x;v_*)$.

For an arbitrary $C\in\mathbb R^{M\times r}$, set
\[
A(t)=[B,tC]
\]
Also, define
\[
F_x(z):=\prod_{\ell=1}^M\sigma\{x_\ell(z_\ell-b_{\ell*})\}
\]
Then,
\[
p_i(x\mid A(t),b_*)=E_{\zeta,\xi}[F_x(B\zeta+tC\xi)].
\]
Since $\nabla F_x$ is bounded and ${\mathbb E}\|\xi\|<\infty$,
we may interchange differentiation and expectation to obtain
\[
\left.\frac{d}{dt}p_i(x\mid A(t),b_*)\right|_{t=0}
=E_{\zeta,\xi}[\nabla F_x(B\zeta)^\top C\xi]
=E_\zeta[\nabla F_x(B\zeta)]^\top C E_\xi[\xi]
=0
\]
Thus, all scores for the last $i-j$ columns are zero.

Therefore, if the parameters are ordered as the first $j$ columns, the difficulties, and the last $i-j$ columns, the score at the true parameter is
\[
s_i(x;w_*)=
\begin{bmatrix}
s_j(x;v_*)\\
0
\end{bmatrix}
\]
Hence,
\[
I_*=\sum_xq_xs_i(x;w_*)s_i(x;w_*)^\top
=
\begin{bmatrix}
\displaystyle\sum_xq_xs_j(x;v_*)s_j(x;v_*)^\top&0\\
0&0_{M(i-j)}
\end{bmatrix}
=
\begin{bmatrix}
J_*&0\\
0&0_{M(i-j)}
\end{bmatrix}.
\]
It follows that $\operatorname{rank}I_*=\operatorname{rank}J_*$.
\end{proof}

\section*{Appendix B: Proof of Lemma~\ref{lem:kg-gram-image}}

\begin{proof}
The implication $\Longrightarrow$ follows from $B^\top N=0$, since
$$N^\top(DB^\top+BD^\top)N=N^\top D\cdot B^\top N+(B^\top N)^\top\cdot D^\top N=0$$
To prove $\Longleftarrow$, set $G:=B^\top B$ and $\Pi:=BG^{-1}B^\top$. The matrix $G$ is nonsingular, and $\Pi$ is the orthogonal projection onto the column space of $B$. Since the columns of $N$ span its orthogonal complement, $N^\top HN=0$ is equivalent to
\begin{equation}\label{eq701}
(I_M-\Pi)H(I_M-\Pi)=0
\end{equation}
Now define
$$
D:=\left(I_M-\frac12\Pi\right)HBG^{-1}
$$
Then,
$$
\begin{aligned}
DB^\top+BD^\top&=
(I_M-\frac12\Pi)HBG^{-1}B^\top+B(G^{-1})^\top B^\top H^\top (I_M-\frac{1}{2}\Pi)
=
(I_M-\frac{1}{2}\Pi)H\Pi +\Pi H^\top (I_M-\frac{1}{2}\Pi)\\
&=
H\Pi+\Pi H-\Pi H\Pi=H
\end{aligned}
$$
The last equality follows from Eq.~\eqref{eq701}.

Next, suppose that
\begin{equation}\label{eq702}
DB^\top+BD^\top=0
\end{equation}
Multiplying on the left by $I_M-\Pi$ and on the right by $BG^{-1}$, and using
$(I_M-\Pi)B=B-B=0$, gives
$$0=(I_M-\Pi)(DB^\top +BD^\top)BG^{-1}
=(I_M-\Pi)DB^\top BG^{-1}+(I_M-\Pi) B\cdot D^\top BG^{-1}
=(I_M-\Pi)D
$$
Thus, we can write $D=B\Omega$. Substituting $D=B\Omega$ into Eq.~\eqref{eq702} and multiplying on the left by $G^{-1}B^\top$ and on the right by $BG^{-1}$ gives
$\Omega+\Omega^\top=0$
Conversely, for any skew-symmetric matrix $\Omega$, $D=B\Omega$ satisfies Eq.~\eqref{eq702}.
\end{proof}

\end{document}